\documentclass[12pt,reqno]{amsart} %12pt
\usepackage[T1]{fontenc}
\usepackage{amsfonts}
\usepackage{amssymb}
\usepackage{mathrsfs}
\usepackage[latin1]{inputenc}
\usepackage{amsmath}
\usepackage{paralist}
\usepackage[english]{babel}
\usepackage{setspace}
\usepackage{bbm}

\newtheorem{theorem}{Theorem}[section]
\newtheorem{lemma}[theorem]{Lemma}

\newtheorem{remark}[theorem]{Remark}
\newtheorem{prop}[theorem]{Proposition}

\numberwithin{equation}{section}
\newcommand{\R}{{\mathbb R}}
\newcommand{\D}{{\mathbb D}}

\newcommand{\C}{{\mathbb C}}
\newcommand{\N}{{\mathbb N}}
\newcommand{\T}{{\mathbb T}}
\newcommand{\cL}{{\mathcal L}}

\newcommand{\ve}{\varepsilon}

\newcommand{\be}{\beta}

\newcommand{\si}{\sigma}
\newcommand{\su}{\subseteq}

\newcommand\Ker{\mathop{\rm Ker}}

\newcommand{à}{\`a}

\begin{document}

\title{Generalized Ces\`aro operators in the disc algebra and in   Hardy spaces}

\author{Angela\,A. Albanese, Jos\'e Bonet and Werner\,J. Ricker}%A.\,A. Albanese\textsuperscript{*}, J. Bonet\textsuperscript{+} and W.\,J. Ricker}

%46E10, 47A10, 47A16, 47A35
%\thanks{\textsuperscript{*} Support of the Alexander von Humboldt Foundation is gratefully acknowledged.}
%\thanks{\textsuperscript{*} Research partially supported by    MEC andFEDER Project MTM 2007-62643, GV Project Prometeo/2008/101 and the net MTM 2007--30904--E (Spain).}
%, compact operator, spectrum, supercyclic operator, mean ergodic operator, power bounded operator
%\providecommand{\sternchen}{\textsuperscript{*}}
\thanks{Article accepted for publication in Advances in Operator Theory.}

\address{ Angela A. Albanese\\
Dipartimento di Matematica e Fisica
``E. De Giorgi''\\
Universit\`a del Salento- C.P.193\\
I-73100 Lecce, Italy}
\email{angela.albanese@unisalento.it}

\address{Jos\'e Bonet \\
Instituto Universitario de Matem\'{a}tica Pura y Aplicada
IUMPA \\
Edificio IDI5 (8E), Cubo F, Cuarta Planta \\
Universitat Polit\`ecnica de Val\`encia \\
E-46071 Valencia, Spain} \email{jbonet@mat.upv.es}

\address{Werner J.  Ricker \\
Math.-Geogr. Fakultät \\
 Katholische Universität
Eichst\"att-Ingol\-stadt \\
D-85072 Eichst\"att, Germany}
\email{werner.ricker@ku.de}

\begin{abstract}
Generalized Cesàro operators $C_t$, for $t\in [0,1)$, are investigated when they act on the disc algebra  $A(\D)$ and on the Hardy spaces $H^p$, for $1\leq p \leq \infty$. We study the continuity, compactness, spectrum and point spectrum of $C_t$ as well as their linear dynamics and mean ergodicity on these spaces.
\end{abstract}

\maketitle

\markboth{A.\,A. Albanese, J. Bonet and W.\,J. Ricker}%
{\MakeUppercase{Generalized Ces\`aro operators}}

\textit{Dedicated to the memory of Prof. Albrecht Pietsch}

\section{Introduction and Preliminaries}

The (discrete) generalized Cesàro operators $C_t$, for $t\in [0,1]$, were first investigated by Rhaly, \cite{R1,R2}. The action of $C_t$  from the sequence space $\omega:=\C^{\N_0}$ into itself, with $\N_0:=\{0,1,2,\ldots\}$, is given by
\begin{equation}\label{Ces-op}
	C_tx:=\left(\frac{t^nx_0+t^{n-1}x_1+\ldots +x_n}{n+1}\right)_{n\in\N_0},\quad x=(x_n)_{n\in\N_0}\in\omega.
	\end{equation}
For $t=0$ and with $\varphi:=(\frac{1}{n+1})_{n\in\N_0}$ note that $C_0$ is the diagonal operator
\begin{equation}\label{Dia-op}
	D_\varphi x:= \left(\frac{x_n}{n+1}\right)_{n\in\N_0}, \quad x=(x_n)_{n\in\N_0}\in\omega,
	\end{equation}
and, for $t=1$, that $C_1$ is the classical Ces\`aro averaging operator
\begin{equation}\label{Ces-1}
	C_1x:=\left(\frac{x_0+x_1+\ldots+x_n}{n+1}\right)_{n\in\N_0},\quad x=(x_n)_{n\in\N_0}\in\omega.
\end{equation}

The behaviour of $C_t$ on certain sequence spaces is well understood; see
 \cite{ABR-N,CR4,S-E}, and the references therein.

 By $H(\D)$ we denote the space of all holomorphic functions on  $\D:=\{z\in\C\,:\, |z|<1\}$, which is
 equipped with the topology $\tau_c$ of uniform convergence on the compact subsets of  $\D$. According to \cite[\S 27.3(3)]{23} the space $H(\D)$ is a Fr\'echet-Montel space. A family of norms generating $\tau_c$ is given, for each $0<r<1$, by
\begin{equation}\label{eq.norme-sup}
	q_r(f):=\sup_{|z|\leq r}|f(z)|,\quad f\in H(\D).
\end{equation}
We identify  a function $f\in H(\D)$ with its sequence of Taylor coefficients $\hat{f}:=(\hat{f}(n))_{n\in\N_0}$ (i.e., $\hat{f}(n):=\frac{f^{(n)}(0)}{n!}$,  for $n\in\N_0$),  so that $f(z)=\sum_{n=0}^\infty \hat{f}(n)z^n$, for $z\in\D$.
%The following integral representation of the generalized Cesàro operators $C_t$ when defined on $H(\D)$, for $t\in [0,1)$, was given in \cite{ABR-NN}. Namely,  define
The operator $C_t\colon H(\D)\to H(\D)$, for $t\in [0,1)$,   given by $ (C_tf)(0):=f(0)$ and
\begin{equation}\label{eq.formula-int}
(C_tf)(z):=\frac{1}{z}\int_0^z\frac{f(\xi)}{1-t\xi}\,d\xi,\ z\in \D\setminus\{0\},
\end{equation}
for every $f \in H(\D)$, was investigated in \cite{ABR-NN}.  For $f(z)=\sum_{n=0}^\infty \hat{f}(n)z^n$ in $H(\D)$, we also have that
\begin{align}\label{eq.rapp-serie}
	(C_tf)(z)=\sum_{n=0}^\infty\left(\frac{t^n\hat{f}(0)+t^{n-1}\hat{f}(1)+\ldots +\hat{f}(n)}{n+1}\right)z^{n}.
\end{align}
Note that the coefficients in \eqref{eq.rapp-serie} are  as in \eqref{Ces-op}.

The discrete generalized Cesàro operator acting in $\omega$ (cf. \eqref{Ces-op}) is denoted by
 $C_t^\omega$, whereas  the notation $C_t$ will be used for the operator \eqref{eq.formula-int} acting in $H(\D)$. Note that $C_0^\omega=D_\varphi$ (see \eqref{Dia-op}). Moreover, for $t=0$ observe that the operator $(C_0f)(z)=\frac{1}{z}\int_0^z f(\xi)\,d\xi$ for $z\not=0$ and $(C_0f)(0)=f(0)$ is the traditional Hardy operator in $H(\D)$.

 The  generalized Cesàro operators $C_t$, for $t\in [0,1]$, are  investigated in \cite{ABR-NN}, where they act on the Fr\'echet space $H(\D)$  and on the weighted Banach spaces $H_v^\infty$ and $H_v^0$ endowed with with their sup-norms. The operator $C_t$ is actually continuous on $H(\D)$,  \cite[Proposition 2.1]{ABR-NN}. In  this article
 we study the operators $C_t$ when they act   on the disc algebra  $A(\D)$ and on the Hardy spaces $H^p$ over $\D$, for $1\leq p\leq\infty$.

 Section 2 establishes the continuity and  compactness of $C_t$ acting on the spaces $A(\D)$ and $H^\infty$ and  determines their spectrum. It is also shown that each $C_t$, for $t\in [0,1)$, is power bounded and uniformly mean ergodic in these spaces but, it fails to be supercyclic. In Section 3 the continuity, compactness and spectrum of $C_t$ are investigated when they act in the Hardy spaces $H^p$ for $1\leq p<\infty$. Estimates for the operator norm $\|C_t\|_{H^p\to H^p}$ are also provided. To establish the compactness of $C_t$ in $H^p$ it is necessary to invoke properties of the forward and backward shift operators acting in $H^p$ as well as properties of certain   Volterra type operators. An important point is that each $C_t$, for $t\in (0,1)$, belongs to the class of Volterra operators  being considered. It is also established that $C_t$ maps $H^\infty$ into $A(\D)$. The spectrum and point spectrum of $C_t$ (acting on $H^p$) are completely determined. As in the case when $C_t$ acts in $A(\D)$ or $H^\infty$, it turns out that $C_t$ is again power bounded and uniformly mean ergodic when acting in $H^p$, $1\leq p<\infty$, but fails to be supercyclic.

  Let us briefly recall the definition of the spaces involved; see \cite{Du,Z} for more details.
The space $H^\infty$ is a Banach space when it is endowed with the norm
\begin{equation}\label{eq.supnorm}
	\|f\|_\infty:=\sup_{z\in\D}|f(z)|,\quad f\in H^\infty.
\end{equation}
 	The disc algebra $A(\D)$  consists of all holomorphic functions on  $\D$
 	 that extend to a continuous function on the closure $\overline{\D}$ of $\D$. That is, $f\in A(\D)$ if and only if there exists $\tilde{f}\in C(\overline{\D})$, necessarily unique, such that $f(z)=\tilde{f}(z)$ for all $z\in\D$. In particular, $f\in H^\infty$.
 	  	Endowed with the norm $\|\cdot\|_\infty$ the space $A(\D)$ is a commutative Banach algebra with respect to multiplication of functions. It is routine to verify that  $A(\D)$ is a closed subalgebra of $H^\infty$.
 	
 	 For $0< p<\infty$ the Hardy space $H^p$ consists of all  functions $f\in H(\D)$  satisfying
 	 \begin{equation}\label{eq.pnorm}
 	 	\|f\|_p:=\sup_{0\leq r<1}\left(\frac{1}{2\pi}\int_0^{2\pi}|f(re^{i\theta})|^p\,d\theta\right)^{1/p}<\infty.
 	 	\end{equation}
  	For $p=\infty$, a function $f\in H(\D)$ belongs to $H^\infty$ if
  	\begin{equation}\label{eq.normI}
  		\sup_{0\leq r<1}\left(\sup_{\theta\in [0,2\pi]}|f(re^{i\theta})|\right)<\infty.
  	\end{equation}
  Of course, the expression in \eqref{eq.normI} equals $\|f\|_\infty$; see \eqref{eq.supnorm}.
  	With this quasi-norm,  $H^p$ is a metrizable, complete, topological vector space. For $1\leq p< \infty$, it turns out that $\|\cdot\|_p$ is actually a norm for which $H^p$ is a Banach space.
  	
  	We are mainly interested in the setting when $1\leq p<\infty$. In this case,   for all $1\leq  p\leq  q \leq \infty$, it turns out that $H^q\subseteq H^p$ with a continuous inclusion. Moreover,  the norm  $\|\cdot \|_p$ is increasing with $p$.
  	Given $1\leq p<\infty$ and $f\in H^p$, there is a standard notation for the terms in \eqref{eq.pnorm}, namely
  	\[
  	M_p(r,f):=\left(\frac{1}{2\pi}\int_0^{2\pi}|f(re^{i\theta})|^p\,d\theta\right)^{1/p}, \quad r\in [0,1),
  	\]
  	which increase with $0\leq r\uparrow 1$. Accordingly, for every $r\in [0,1)$, we have
  	\begin{equation}\label{eq.disnorm}
  		M_p(r,f)\leq \|f\|_p=\sup_{0\leq r<1}M_p(r,f)=\lim_{r\to1^-}M_p(r,f).
  	\end{equation}
  The same is true for $p=\infty$ if we define
  \[
  M_\infty (r,f):=\sup_{\theta\in [0,2\pi]}|f(re^{i\theta})|,\quad r\in [0,1).
  \]
  	It is known, for each $z\in\D$, that the evaluation functional $\delta_z\colon f\mapsto f(z)$, for $f\in H^p$, is continuous, that is, $\delta_z\in (H^p)'$; see the Lemma on p.36  of \cite{Du} for $1\leq p<\infty$. For $p=\infty$, fix $z\in\D$ and let $r:=|z|$. Then the inequality
  	\[
  	|\delta_z(f)|=|f(z)|\leq M_\infty(r,f)\leq \|f\|_\infty,\quad f\in H^\infty,
  	\]
  	shows that $\delta_z\in (H^\infty)'$.
  	
We end this section by recalling a few definitions  and some notation concerning Fr\'echet  spaces (always locally convex) and operators between them. For further details about functional analysis and operator theory relevant to this paper see, for example, \cite{Ed,23,24,Ru}.

Given Fr\'echet spaces $X, Y$,  denote by $\cL(X,Y)$ the space of all  linear operators from $X$ into $Y$ which are continuous. If $X=Y$, then we simply write $\cL(X)$ for $\cL(X,X)$. If both $X$, $Y$ are Banach spaces then, for the operator norm $\|T\|_{X\to Y}:=\sup_{\|x\|_X\leq 1}\|Tx\|_Y$, with $T\in \cL(X,Y)$, the space $\cL(X,Y)$ is a Banach space.
Equipped with the topology of pointwise convergence on a Fr\'echet space $X$ (i.e., the strong operator topology $\tau_s$) the quasi-complete locally convex Hausdorff space $\cL(X)$ is denoted by $\cL_s(X)$. The range $T(X):=\{Tx:\ x\in X\}$ of $T\in \cL(X)$ is also denoted by ${\rm Im}(T)$. Moreover, $\Ker(T):=\{x\in X:\ Tx=0\}$.
 %Equipped with the strong topology $\tau_b$  the Banach space $\cL(X)$ is denoted by $\cL_b(X)$.

Let $X$ be a Fr\'echet space.
 The identity operator on $X$ is written as $I$.  The \textit{transpose operator} of $T\in \cL(X)$ is denoted by  $T'$; it acts from the topological dual space $X':=\cL(X,\C)$ of $X$ into itself. Denote by $X'_\si$ (resp., by $X'_\beta$) the space $X'$ equipped with the weak* topology $\si(X',X)$ (resp., with the strong dual topology $\beta(X',X)$). It is known that $X'_\si$ is quasicomplete, that $T'\in \cL(X'_\si)$ and also that $T'\in \cL(X_\be')$,  \cite[p.134]{24}. The bi-transpose operator $(T')'$ of $T$ is  denoted simply by $T''$ and belongs to $\cL((X'_\beta)'_\beta)$. In the event that $X$ is a Banach space, both $X'_\beta$ (denoted simply by $X'$) and $(X'_\beta)'_\beta$ (denoted simply by $X''$) are again Banach spaces. The dual norm in $X'$ is given $\|x'\|:=\sup_{\|x\|\leq 1}|\langle x,x'\rangle|$, for $x'\in X'$.

The following result, \cite[Theorem 6.4]{ABR-N}, which is  stated below for Banach spaces, will be  needed in the sequel. Given $\delta>0$, we introduce the notation $B(0,\delta):=\{z\in\C:\ |z|<\delta\}$ and $\overline{B(0,\delta)}$ for its closure in $\C$.  Denote the unit circle in $\C$ by $\mathbb{T}:=\{z\in\C :\ |z|=1\}$. Of course, $B(0,1)=\D$ and $\mathbb{T}$ is the boundary of $\D$. For the definition of power bounded and uniformly mean ergodic operators see Section 2.

\begin{theorem}\label{Th-ABR} Let $X$ be a Banach space and let $T\in \cL(X)$ be a compact operator such that $\sigma(T;X)\subseteq \overline{\D}$ and $\sigma(T;X)\cap \T=\{1\}$
	  and which satisfies  $\Ker (I-T)\cap {\rm Im} (I-T)=\{0\}$. Then $T$ is power bounded and uniformly mean ergodic.
\end{theorem}

\begin{remark}\label{R.12}\rm  The conditions in Theorem \ref{Th-ABR} concerning $\sigma(T;X)$  are equivalent to
 	$1\in \sigma(T;X)$ and $\sigma(T;X)\setminus \{1\}\subseteq \overline{B(0,\delta)}$, for some $\delta\in (0,1)$.
 	\end{remark}

\section{Continuity, Compactness and Spectrum of $C_t$  in $H^\infty$ and in $A(\D)$}

%In this section we first establish, for $t\in [0,1)$, the continuity of $C_t\colon A(\D)\to A(\D)$.  The same result is true for    $C_t\colon H^\infty\to H^\infty$ by \cite[Proposition 2.3]{ABR-NN}.
% Both the operators $C_t\in \cL(H^\infty)$  and $C_t\in \cL(A(\D))$ are compact operators (cf. Proposition \ref{Compact}); their spectrum is identified in Proposition \ref{Spectrum}.

For $t=1$ it is clear that
 $C_1$ fails to act from  $H^\infty$ into itself; consider $C_1\mathbbm{1}$, where $\mathbbm{1}$ is the constant function $z\mapsto 1$ for $z\in\D$. For a study of $C_1$ acting on $H^\infty$ see \cite{D-S}
 and, for $C_t$ acting on $H^\infty$,  for $t\in [0,1)$, see  \cite[Proposition 2.3]{ABR-NN}.

%\begin{prop}\label{Cont_H_Infty} For each $t\in [0,1)$ the  operator $C_t\colon H^\infty\to H^\infty$ is continuous. Moreover, the operator norms are given by  $\|C_0\|_{H^\infty\to H^\infty}=1$ and $$
%	\|C_t\|_{H^\infty\to H^\infty}=-\frac{\log (1-t)}{t},\quad t\in (0,1).
%	$$
%\end{prop}

%Denote the unit circle in $\C$  by $\T:=\{z\in\C:\ |z|=1\}$, that is, $\T$ is  the boundary of $\D$.
The following result establishes that $C_t\in \cL(A(\D))$ and identifies its operator norm precisely.

\begin{prop}\label{Cont_A_D} For each $t\in [0,1)$ the  operator $C_t\colon A(\D)\to A(\D)$ is continuous. Moreover, the operator norms are given by  $\|C_0\|_{A(\D)\to A(\D)}=1$ and
	$$
	\|C_t\|_{A(\D)\to A(\D)}=\frac{-\log (1-t)}{t},\quad t\in (0,1).
	$$
\end{prop}

\begin{proof} Fix $t\in [0,1)$. We first show that $C_t (A(\D)) \subseteq A(\D)$. Recall,  for $f\in H(\D)$,  that  we have $(C_tf)(0)=f(0)$ and, via the parametrization $\xi:=sz$, for $s\in [0,1]$, that
	\begin{equation}\label{eq.op_p}
	(C_tf)(z)=\frac{1}{z}\int_0^z\frac{f(\xi)}{1-t\xi}\,d\xi=\int_0^1\frac{f(sz)}{1-stz}\,ds, \quad z\in \D\setminus\{0\}.
	\end{equation}
Moreover, if $f\in H^\infty$ then $C_tf\in H^\infty$,  \cite[Proposition 2.3]{ABR-NN}. Accordingly, for a fixed function $f\in A(\D)\subseteq H^\infty$, it follows that $C_tf\in H^\infty$. The claim is that  $C_tf$ is the restriction to $\D$ of an element of $C(\overline{\D})$. Indeed, let $\tilde{f}\in C(\overline{\D})$ agree with $f$ on $\D$. Then it is routine to  verify that the function $z\mapsto \int_0^1 \frac{\tilde{f}(sz)}{1-stz}\,ds$, for $z\in \overline{\D}$, is continuous on $\overline{\D}$ and agrees with $C_tf$ on $\D$. Accordingly, $C_tf\in A(\D)$.
%	To verify this claim we proceed as follows. Let $\tilde{f}\in C(\overline{\D})$ agree with $f$ on $\D$. This implies that $(C_tf)(z)$ is also well defined whenever $z\in \T$. Indeed, fix $z_0\in \T$. Then  the function $s\mapsto \frac{\tilde{f}(sz_0)}{1-stz_0}$ is continuous in $[0,1]$ and agrees with $s\mapsto \frac{{f}(sz_0)}{1-stz_0}$ on $[0,1)$.
%	Hence, the integral  $\int_0^1 \frac{f(sz_0)}{1-stz_0}\,ds=\int_0^1 \frac{\tilde{f}(sz_0)}{1-stz_0}\,ds$ is well defined. On the other hand, the function $G(s,z):=\frac{\tilde{f}(sz)}{1-stz}$ is continuous on $[0,1]\times \overline{\D}$ and satisfies
%\begin{equation}\label{eq.SG}
%|G(s,z)|\leq\frac{|\tilde{f}(sz)|}{|1-stz|}\leq \frac{\|f\|_\infty}{1-t}, \quad (s,z)\in [0,1]\times \overline{\D},
%\end{equation}
%since $|1-stz|\geq 1-st|z|\geq 1-t$. Now, let $(z_n)_{n\in\N}\subseteq {\D}$ satisfy $z_n\to z_0$ in $\overline{\D}$ as $n\to\infty$. By \eqref{eq.SG} we can apply the dominated convergence theorem to conclude that
%\[
%\int_0^1\frac{{f}(sz_n)}{1-stz_n}\,ds= \int_0^1\frac{\tilde{f}(sz_n)}{1-stz_n}\,ds=\int_0^1 G(s,z_n)\, ds\to \int_0^1 G(s,z_0)\,ds, \mbox{ as } n\to\infty.
%\]
%In view of \eqref{eq.op_p} this implies that $z\mapsto \int_0^1\frac{\tilde{f}(sz)}{1-stz}\,ds$, for $z\in \overline{\D}$,  is a continuous function on $\overline{\D}$ which agrees with $C_tf$ on $\D$. Accordingly, $C_tf\in A(\D)$.
	
Since  $C_t (A(\D)) \subseteq A(\D)$, for $t\in [0,1)$, and $A(\D)$ is a Banach space, by a closed graph argument we can conclude  that $C_t\in \cL(A(\D))$. Indeed, let $f_n\to 0$ in $A(\D)$ and $C_tf_n\to g$ in $A(\D)$ for $n\to\infty$. Since $A(\D)\subseteq H(\D)$ continuously, also $f_n\to 0$ in $H(\D)$ and $C_tf_n\to g$ in $H(\D)$ for $n\to\infty$. The continuity of $C_t\in \cL(H(\D))$ yields that $C_tf_n\to 0$ in $H(\D)$ for $n\to\infty$ and so $g=0$.

 We now compute the norm	of $C_t\in \cL(A(\D))$ for $t\in [0,1)$.

 Consider first $t=0$.
Let $f\in A(\D)$ be fixed. Then $f\in H^\infty$ and so, by  \cite[Proposition 2.3]{ABR-NN}, we can conclude that
	$\|C_0f\|_\infty\leq \|f\|_\infty$.
	Since $A(\D)$ is a closed subspace of $H^\infty$,
	this implies that $\|C_0\|_{A(\D)\to A(\D)}\leq 1$. On the other hand, the function $h_0(z):=1$, for $z\in\overline{\D}$, belongs to $A(\D)$ and satisfies both $\|h_0\|_\infty=1$ and  $C_0h_0=h_0$. It follows  that $\|C_0\|_{A(\D)\to A(\D)}=1$.
	
	Now let $t\in (0,1)$. Again by  \cite[Proposition 2.3]{ABR-NN}, for each $f\in A(\D)\subseteq H^\infty$ we have that
	\begin{equation*}
		\|C_tf\|_\infty\leq \frac{-\log(1-t)}{t}\|f\|_\infty,
	\end{equation*}
which implies that  $\|C_t\|_{A(\D)\to A(\D)}\leq \frac{-\log(1-t)}{t}$. But, $\|C_th_0\|_\infty= \frac{-\log(1-t)}{t}$ and so $\|C_t\|_{A(\D)\to A(\D)}= \frac{-\log(1-t)}{t}$.
\end{proof}

Recall that a linear map $T\colon X\to Y$, with $X$ and $Y$ Banach spaces, is called \textit{compact} if  $T(B_X)$ is a relatively compact set in $Y$, where $B_X$ denotes the closed unit ball  of $X$. It is routine to show that necessarily $T\in \cL(X,Y)$.
%If $X,Y$ are Fr\'echet spaces, then $T$ is compact if there exists an open neighbourhood $\cU_X$ of $0\in X$ such that $T(\cU_X)$ is a relatively compact set in $Y$.
 For the spectral theory for compact operators see \cite{Ed,Gr,Ru}, for example.

The proof of the next result, which establishes the compactness of $C_t$ on both $A(\D)$ and $H^\infty$, proceeds  along the lines of the proof of \cite[Proposition 2.7]{ABR-NN}.

\begin{prop}\label{Compact}    For each $t\in [0,1)$, both  of the operators $C_t\colon H^\infty\to H^\infty$ and $C_t\colon A(\D)\to A(\D)$ are  compact.
\end{prop}

\begin{proof} Fix $t\in [0,1)$. Since $A(\D)$ is a closed subspace of $H^\infty$ and $C_t(A(\D))\subseteq A(\D)$ (cf. Proposition \ref{Cont_A_D}),
	it suffices to show that $C_t\colon H^\infty\to H^\infty$ is compact.
	
	First we establish the following Claim:
	\begin{itemize}
		\item[($\star$)] Let a sequence $(f_n)_{n\in\N}\subset H^\infty$ satisfy $\|f_n\|_{\infty}\leq 1$ for every $n\in\N$ and $f_n\to 0$ in $(H(\D),\tau_{c})$ for $n\to\infty$. Then $C_tf_n\to 0$ in $H^\infty$.
	\end{itemize}
To prove the Claim, let $(f_n)_{n\in\N}\subset H^\infty$ be a sequence as in ($\star$).
Fix $\varepsilon>0$ and select $\delta\in (0,\beta)$, where $\beta:=\min\{1,\frac{\ve(1-t)}{2}\}$. Since $\overline{B(0,1-\delta)}$ is a compact subset of $\D$, there exists $n_0\in\N$ such that
\[
\max_{|\xi|\leq 1-\delta}|f_n(\xi)|<\delta, \quad n\geq n_0.
\]
Recall that $(C_tf_n)(0)=f_n(0)$ for every $n\in\N$. Therefore, $(C_tf_n)(0)\to 0$ as $n\to\infty$. For $z\in \D\setminus\{0\}$ we have, via \eqref{eq.op_p}, that
\begin{equation*}
	|(C_tf_n)(z)|=\left|\int_0^1\frac{f_n(sz)}{1-stz}ds\right|\leq \int_0^{1-\delta}\frac{|f_n(sz)|}{|1-stz|}ds+\int_{1-\delta}^1\frac{|f_n(sz)|}{|1-stz|}ds.
\end{equation*}
Denote the first (resp., second) summand in the right-side of the previous  inequality by $(A_n)$ (resp., by $(B_n)$).
Since $|1-stz|\geq 1-st |z|\geq \max\{1-s,1-t,1-|z|\}$, for all $s,t\in [0,1)$ and $z\in \D$, it follows, for every $n\geq n_0$, that $\int_0^{1-\delta}|f_n(sz)|\,ds\leq (1-\delta)\max_{|\xi|\leq (1-\delta)}|f_n(\xi)|$ (as $|sz|\leq (1-\delta)$ for all $s\in [0,1-\delta]$) and hence, that
\[
(A_n)\leq  \frac{(1-\delta)}{1-t}\max_{|\xi|\leq 1-\delta}|f_n(\xi)|<\frac{\ve}{2}.
\]
On the other hand, for every $n\geq n_0$, observe that
\[
(B_n)=\int_{1-\delta}^1 \frac{|f_n(sz)|}{|1-stz|}\,ds\leq\int_{1-\delta}^1\frac{\|f_n\|_{\infty}}{1-t}\,ds\leq \frac{\delta}{1-t}<\frac{\ve}{2}.
\]
It follows that $\|C_tf_n\|_{\infty}<\ve$ for every $n\geq n_0$. That is, $C_tf_n\to 0$ in $H^\infty$ for $n\to\infty$ and so ($\star$) is proved.

Concerning the compactness of $C_t\in\cL(H^\infty)$, let $(f_n)_{n\in\N}\subset H^\infty$ be any bounded sequence. There is no loss of generality in assuming that $\|f_n\|_{\infty}\leq 1$ for all $n\in\N$. To establish the compactness of $C_t\in\cL(H^\infty)$  we need to show that $(C_tf_n)_{n\in\N}$ has a convergent subsequence in $H^\infty$. This follows routinely from the fact that $(f_n)_{n\in\N}$ is also bounded in the Fr\'echet-Montel space  $H(\D)$, as $H^\infty\su H(\D)$ continuously, together with condition ($\star$).
%
%Since $H^\infty\su H(\D)$ continuously, the sequence $(f_n)_{n\in\N}$ is also bounded in the Fr\'echet-Montel space $H(\D)$. Hence, there is a subsequence $g_j:=f_{n_j}$, for $j\in\N$, of $(f_n)_{n\in\N}$ and $f\in H(\D)$ such that $g_j\to f$ in $H(\D)$ with respect to $\tau_c$. In particular, $g_j\to f$ pointwise on $\D$. Since $|g_j(z)|=|f_{n_j}(z)|\leq 1$ for all $z\in\D$ and  $j\in\N$, letting $j\to \infty$ it follows that $|f(z)|\leq 1$ for all $z\in\D$, that is, $f\in H^\infty$ with $\|f\|_{\infty}\leq 1$. Let $h_j:=\frac{1}{2}(g_j-f)$, for $j\in\N$. Then $\|h_j\|_{\infty}\leq 1$, for $j\in\N$, and $h_j\to 0$ in $H(\D)$ with respect to $\tau_c$.  Condition (*) implies that $C_th_j\to 0$ in $H^\infty$ from which it follows that $C_tf_{n_j}=C_tg_j=C_t(g_j-f)+C_tf=2C_th_j+C_tf\to C_tf $ in $H^\infty$, as desired.
\end{proof}

Given a Fr\'echet space $X$ and $T\in \cL(X)$, the resolvent set $\rho(T;X)$ of $T$ consists of all $\lambda\in\C$ such that the inverse operator  $R(\lambda,T):=(\lambda I-T)^{-1}$ exists in $\cL(X)$. The set $\sigma(T;X):=\C\setminus \rho(T;X)$ is called the \textit{spectrum} of $T$. The \textit{point spectrum}  $\sigma_{pt}(T;X)$ of $T$ consists of all $\lambda\in\C$ (also called an eigenvalue of $T$) such that $(\lambda I-T)$ is not injective. In the event that $X$ is a Banach space, the residual spectrum $\sigma_r(T;X)$ (resp. continuous spectrum $\sigma_c(T;X)$) of $T$ consists of all $\lambda\in\C$ such that $\lambda I-T$ is injective and ${\rm Im}(\lambda I-T)$ is not dense (resp. is proper and dense) in $X$. This provides the pairwise disjoint decomposition
\[
\sigma(T;X)=\sigma_{pt}(T;X)\cup \sigma_r(T;X)\cup \sigma_c(T;X).
\]

\begin{prop}\label{Spectrum}  For each $t\in [0,1)$ the  spectra of $C_t\in \cL(H^\infty)$ and of $C_t\in \cL(A(\D))$ are given by
	\begin{equation}\label{Sp-pt}
		\sigma_{pt}(C_t;H^\infty)=\sigma_{pt}(C_t;A(\D))=\left\{\frac{1}{m+1}\,: m\in\N_0\right\},
	\end{equation}
and
\begin{equation}\label{Sp}
		\sigma(C_t;H^\infty)=\sigma(C_t;A(\D))=\left\{\frac{1}{m+1}\,: m\in\N_0\right\}\cup\{0\}.
\end{equation}
\end{prop}

\begin{proof} Let $t\in [0,1)$ be fixed.
	By \cite[Lemma 3.6]{CR4}  the point spectrum of the  operator $C_t^\omega\in \cL(\omega)$ is given by $\sigma_{pt}(C_t^\omega;\omega)=\{\frac{1}{m+1}\,:\, m\in\N_0\}$ and, for each $m\in\N_0$,   the corresponding eigenspace $\Ker(\frac{1}{m+1}I-C_t^\omega)$ is 1-dimensional and is generated by an eigenvector $x^{[m]}=(x_n^{[m]})_{n\in\N_0}\in \ell^1$.
	Since $A(\D)\subseteq H^\infty\subseteq H(\D)$ with continuous inclusions and the map $\Phi\colon H(\D)\to  \omega$ given by $f\mapsto (\hat{f}(n))_{n\in\N_0}$, is a continuous embedding (see
	Section 1 of \cite{ABR-NN}), we have that $\sigma_{pt}(C_t;A(\D))\subseteq \sigma_{pt}(C_t;H^\infty)\subseteq \{\frac{1}{m+1}\,:\, m\in\N_0\}$. Indeed, let $f\in H(\D)\setminus\{0\}$ and $\lambda\in \C$ satisfy $C_tf=\lambda f$. Then $\lambda f(z)=\sum_{n=0}^\infty \widehat{(\lambda f)}(n)z^n=\sum_{n=0}^\infty\lambda  \hat{f}(n)z^n$ and, by \eqref{eq.rapp-serie}, we have that $(C_tf)(z)=\sum_{n=0}^\infty (C_t^\omega \hat{f})_nz^n$. It follows that $C_t^\omega \hat{f}=\lambda \hat{f}$ in $\omega$ with $\hat{f}\not=0$ and so $\lambda\in \sigma_{pt}(C_t^\omega;\omega)=\{\frac{1}{m+1}\ :\ m\in\N_0\}$.
	
	To conclude the proof, it remains to show that $\{\frac{1}{m+1}\,:\, m\in\N_0\}\subseteq \sigma_{pt}(C_t;A(\D))$. To establish this recall, for each $m\in\N_0$, that the eigenvector $x^{[m]}$ corresponding to $\frac{1}{m+1}$ belongs to $\ell^1$ and hence, the function $g_m(z):=\sum_{n=0}^\infty x^{[m]}_nz^n$ belongs to $A(\D)$. %Indeed, for $m\in\N_0$ given and each $ n\in\N_0$ we have that
%	\[
%	\sup_{z\in\overline{\D}}|x^{[m]}_nz^n|=|x^{[m]}_n|,
%	\]
%	with $x^{[m]}\in \ell^1$. So, we can apply the Weierstrass M-test to conclude that the series $g_m(z):=\sum_{n=0}^\infty x^{[m]}_nz^n$  converges uniformly in $\overline{\D}$. This implies that  $g_m\in A(\D)$. Moreover, 	$0\leq |g_m(z)|\leq \|x^{[m]}\|_{\ell^1}$ for $z\in\D$.In addition,
	 Then \eqref{eq.formula-int} and \eqref{eq.rapp-serie} imply, for each $z\in\D$, that
	\begin{equation}
	(C_tg_m)(z)=\sum_{n=0}^{\infty}(C_t^\omega x^{[m]})_nz^n=\frac{1}{m+1}\sum_{n=0}^\infty x^{[m]}_nz^n=\frac{1}{m+1}g_m(z).
	\end{equation}
	Thus $g_m$ is an eigenvector of $C_t\in \cL(A(\D))$ corresponding to the eigenvalue $\frac{1}{m+1}$.

	Finally,  $\sigma(C_t;A(\D))=\sigma(C_t;H^\infty)=\{\frac{1}{m+1}\,:\, m\in\N_0\}\cup\{0\}$ follows from the fact that $C_t$ is a  compact operator  on both $A(\D)$ and $H^\infty$.
\end{proof}

An operator $T\in \cL(X)$, with $X$ a Banach space, is called \textit{power bounded} if  $\sup_{n\in\N_0}\|T^n\|_{X\to X}<\infty$. Given $T\in \cL(X)$, define its Cesàro means by
\[
T_{[n]}:=\frac{1}{n}\sum_{m=1}^nT^m,\quad n\in\N.
\]
 Then  $T$ is said to be \textit{mean ergodic} (resp., \textit{uniformly mean ergodic}) if $(T_{[n]})_{n\in\N}$ is a convergent sequence in $\cL_s(X)$ (resp., a convergent sequence for the operator norm in  $\cL(X)$). It is routine to check that $\frac{T^n}{n}=T_{[n]}-\frac{n-1}{n}T_{[n-1]}$, for $n\geq 2$, and hence,  $\tau_s$-$\lim_{n\to\infty}\frac{T^n}{n}=0$ whenever $T$ is mean ergodic.  An operator  $T\in \cL(X)$ is said to be \textit{supercyclic} if, for some $z\in X$, the projective orbit $\{\lambda T^nz\ :\ \lambda\in\C,\ n\in\N_0\}$ is dense in $X$. Since the closure of the linear span of a projective orbit is separable, whenever any supercyclic operator in $\cL(X)$ exists, then $X$ is necessarily  separable.

For the linear dynamics of $T$ we refer to \cite{B-M,G-P} and for mean ergodic operators to \cite{K}, for example.

As a consequence of
 the previous proposition, combined with Theorem \ref{Th-ABR}, we have the following result.

\begin{prop}\label{Dyn-Hv} For each $t\in [0,1)$ both of  the operators $C_t\in \cL(H^\infty)$ and  $C_t\in \cL(A(\D))$ are power bounded, uniformly mean ergodic but, fail to be supercyclic.
\end{prop}

\begin{proof} Fix $t\in [0,1)$.  Since $H^\infty$ is non-separable,  $C_t\in \cL(H^\infty)$ cannot be supercyclic.
	
	 The operator $C_t$ is compact on  both  $H^\infty$ and on $A(\D)$ (cf. Proposition \ref{Compact}). Therefore, the compact transpose operators $C_t'\in \cL((H^\infty)')$ and $C'_t\in \cL((A(\D))')$ have the same non-zero eigenvalues as $C_t$ (see, e.g., \cite[Theorem 9.10-2(2)]{Ed}). In view of Proposition \ref{Spectrum} it follows that $\sigma_{pt}(C_t';(H^\infty)')=\sigma_{pt}(C_t';(A(\D))')=\{\frac{1}{m+1}\,:\, m\in\N_0\}$. We can apply \cite[Proposition 1.26]{B-M} to conclude that $C_t$ is not supercyclic on  the separable space $A(\D)$.
	
	By Proposition \ref{Spectrum} and its proof (as $x^{[0]}=(t^n)_{n\in\N_0}$) we have that $\Ker (I-C_t)={\rm span}\{h_t\}$, with $h_t(z):=\sum_{n=0}^\infty t^nz^n=\frac{1}{1-tz}$, for $z\in B(0,\frac{1}{t})$ (with $\frac{1}{t}>1$). Note that $\overline{\D}\subseteq  B(0,\frac{1}{t})$. The function $h_t\in A(\D)$ is the eigenvector denoted by $g_0$ in the proof of Proposition \ref{Spectrum} corresponding to the eigenvalue $1$. On the other hand, ${\rm Im}(I-C_t)$ is a closed subspace of $H^\infty$ (resp., of $A(\D)$), as $C_t$ is compact in $H^\infty$ (resp., in $A(\D)$)). Also,  ${\rm Im}(I-C_t)\subseteq \{g\in H^\infty\,:\, g(0)=0\}$ (resp., ${\rm Im}(I-C_t)\subseteq \{g\in A(\D)\,:\, g(0)=0\}$), because $(C_tf)(0)=f(0)$ for each $f\in H^\infty$ (resp., for each $f\in A(\D)$). Moreover,  \cite[Theorem 9.10.1]{Ed} implies that ${\rm codim}\,{\rm Im}(I-C_t)={\rm dim}\Ker (I-C_t)=1$ regarded as subspaces of $H^\infty$ (resp., of $A(\D)$). Accordingly, both ${\rm Im}(I-C_t)$ and $\{g\in H^\infty\, :\, g(0)=0\}=\Ker (\delta_0)$ are hyperplanes of $H^\infty$ (resp., of $A(\D)$). %where $\delta_0\in (H^\infty)'$ (where $\delta_0\in (A(\D))'$) is the linear evaluation functional $f\mapsto f(0)$, for $f\in H^\infty$ (resp., $f\in A(\D)$).
	It follows that necessarily ${\rm Im}(I-C_t)=\{g\in H^\infty\, :\, g(0)=0\}$ (resp., ${\rm Im}(I-C_t)=\{g\in A(\D)\, :\, g(0)=0\}$).
	
	Let $u\in {\rm Im}(I-C_t)\cap\Ker (I-C_t)$. Then $u(0)=0$ and there exists $\lambda\in \C$ such that $u=\lambda h_t$. This yields that $0=u(0)=\lambda h_t(0)=\lambda$. Hence, $u=0$. So, ${\rm Im}(I-C_t)\cap\Ker (I-C_t)=\{0\}$.
	
	Proposition \ref{Spectrum} implies that $1\in \sigma(C_t; H^\infty)=\sigma(C_t; A(\D))=\{\frac{1}{m+1}\,;\, m\in\N_0\}\cup\{0\}$. Consequently, for $\delta=\frac{1}{2}$, all the assumptions of Theorem \ref{Th-ABR} (see also Remark \ref{R.12}) are satisfied. So, we can conclude that $C_t$ is power bounded and uniformly mean ergodic   both on $H^\infty$ and on $A(\D)$.
\end{proof}

\section{Continuity, Compactness and Spectrum of $C_t$ acting in $H^p$ }

In this section we investigate various properties of the operators $C_t$ when they act in the family of Hardy spaces $H^p$, $1\leq p<\infty$, as well as some related Volterra type operators. We recall a known fact which follows easily from Jensen's inequality  for $\R$-valued functions, \cite{Ru-0}.

\begin{lemma}\label{L.3.1} Let $h\colon [0,1]\to\C$ be an integrable function. For each $p\geq 1$ we have
	\[
	\left|\int_0^1h(s)\,ds \right|^p\leq \int_0^1 |h(s)|^pds.
	\]
\end{lemma}

%\begin{proof}
%	Fix $p\geq 1$. Choose a rotation $c:=e^{i\theta}$ of $\C$ such that $\left|\int_0^1h(s)\,ds \right|=c\int_0^1h(s)\,ds$. Then
%	\begin{eqnarray*}
%	\left| \int_0^1 h(s)\,ds\right|&={\rm Re}\left(c\int_0^1h(s)\,ds\right)=\int_0^1{\rm Re}(ch(s))\,ds\\
%	&\leq \int_0^1 |ch(s)|ds=\int_0^1 |h(s)|ds.
%	\end{eqnarray*}
%Apply Jensen's inequality, \cite{Ru-0}, to the convex function $\varphi(x):=x^p$, for $x\in [0,\infty)$, and the integrable function $|h|\geq 0$ yields
%\[
%	\left(\int_0^1|h(s)|\,ds \right)^p\leq \int_0^1 |h(s)|^pds,
%\]
%from which the desired inequality follows.
%\end{proof}
It follows from \cite[Theorem 1]{AleSi1}, by substituting for $g$ there  the particular function $g(z)=z$, for $z\in\D$, that the Hardy operator $C_0$ is bounded in $H^p$ for all $1\leq p<\infty$. We include a simple and direct proof of this fact based on Lemma \ref{L.3.1}.

\begin{prop}\label{Cont_C0_Hp} Let $1\leq p<\infty$.  The Hardy operator $C_0\colon H^p\to H^p$ is continuous with operator norm $\|C_0\|_{H^p\to H^p}=1$.
\end{prop}
\begin{proof} Recall from \eqref{eq.op_p} that
 the Hardy operator $C_0\colon H(\D)\to H(\D)$ is  given by $(C_0f)(0)=f(0)$ for $z=0$ and, for $z\in\D\setminus\{0\}$, by
	\begin{equation}\label{eq.delta}
	(C_0f)(z)=\frac{1}{z}\int_0^z f(\xi)\,d\xi=\int_0^1 f(sz)\,ds .
	\end{equation}
	For $z\in\D$ fixed,  $s\mapsto f(sz)$ is continuous on $[0,1]$ and hence, integrable.
	For every $f\in H^p$ and $r\in [0,1)$,  Lemma \ref{L.3.1} and Fubini's theorem imply that
	\begin{align*}
	M_p(r,C_0f)^p&=\frac{1}{2\pi}\int_0^{2\pi}|(C_0f)(re^{i\theta})|^p\,d\theta=\frac{1}{2\pi}\int_0^{2\pi}\left|\int_0^1f(sre^{i\theta})\,ds\right|^pd\theta\\
	&\leq \frac{1}{2\pi}\int_0^{2\pi}\left(\int_0^1|f(sre^{i\theta})|^p\,ds\right)\,d\theta=\int_0^1\left(\frac{1}{2\pi}\int_0^{2\pi}|f(sre^{i\theta})|^p\,d\theta\right)ds\\
	&= \int_0^1M_p(sr,f)^pds\leq M_p(r,f)^p\leq \|f\|^p_p.
	\end{align*}
This implies that $C_0f\in H^p$ and $\|C_0f\|_p\leq \|f\|_p$ for every $f\in H^p$. So, $C_0\in \cL(H^p)$ and $\|C_0\|_{H^p\to H^p}\leq 1$. Since $h_0:=\mathbbm{1}\in H^p$ with $\|h_0\|_p=\|\mathbbm{1}\|_p=1$ and $C_0h_0=\mathbbm{1}=h_0$, it follows that $\|C_0\|_{H^p\to H^p}=1$.
\end{proof}

\begin{prop}\label{Cont_Hp} Let $1\leq p<\infty$. For each $t\in (0,1)$ the operator $C_t\colon H^p\to H^p$ is continuous and its operator norm satisfies
	\begin{equation}\label{NO}
	1\leq \|C_t\|_{H^p\to H^p}\leq \frac{1}{1-t},\quad t\in (0,1).
	\end{equation}
Actually, for $p=1$, it is the case that
\begin{equation}\label{norn-Hp1}
	 \|C_t\|_{H^1\to H^1}\leq \frac{-\log (1-t)}{t},
\end{equation}
and for $1 < p < \infty,$ it is the case that
\begin{equation}\label{norn-Hp}
		 \|C_t\|_{H^p\to H^p}\leq \left[\frac{1}{t(p-1)}\left(\frac{1}{(1-t)^{p-1}}-1\right)\right]^{1/p}.
	\end{equation}
\end{prop}

\begin{proof}
	Fix $p\geq 1$ and  $t\in (0,1)$. The function $h_t(z)=\frac{1}{1-tz}$, for $z\in\D$, belongs to  $ A(\D)\subseteq H^\infty$ (see the proof of Proposition \ref{Dyn-Hv}) and satisfies $\|h_t\|_\infty=\frac{1}{1-t}$, because $|h_t(z)|\leq \frac{1}{1-t}$ and $\lim_{z\to 1^-}h_t(z)=\frac{1}{1-t}$. Since $h_t\in H^\infty$, the multiplication operator $M_t\colon H^p\to H^p$ defined by
	\[
	M_tf:=h_tf,\quad f\in H^p,
	\]
	is continuous with $\|M_t\|_{H^p\to H^p}\leq \|h_t\|_{\infty}=\frac{1}{1-t}$, as  is well known (and routine to verify).
	
Since $C_t=C_0\circ M_t$ (cf. \eqref{eq.op_p} and \eqref{eq.delta}), it follows from Proposition \ref{Cont_C0_Hp} that $C_t\in \cL(H^p)$ and that $\|C_t\|_{H^p\to H^p}\leq \|C_0\|_{H^p\to H^p}\cdot \|M_t\|_{H^p\to H^p}=\frac{1}{1-t}$.

On the other hand, the function $h_t$ belongs to $H^p$ and satisfies $C_th_t=h_t$. This yields that $1 \leq \|C_t\|_{H^p\to H^p}$. Hence, \eqref{NO} has been  established.

Concerning \eqref{norn-Hp1} let $p=1$. Then, for every $f\in H^1$ and $r\in [0,1)$, by \eqref{eq.disnorm} and  Fubini's theorem we obtain
	 \begin{align*}
	 	M_1(r,C_tf)=&\frac{1}{2\pi}\int_0^{2\pi}\left|\int_0^1\frac{f(sre^{i\theta})}{1-srte^{i\theta}}ds\right|d\theta\leq \frac{1}{2\pi}\int_0^{2\pi}\left(\int_0^1\frac{|f(sre^{i\theta})|}{1-ts}ds\right)d\theta\\
	 	=&\int_0^1\frac{1}{1-ts}\left(\frac{1}{2\pi}\int_0^{2\pi}|f(sre^{i\theta})|d\theta\right)ds=\left(\frac{-\log(1-t)}{t}\right)M_1(r,f)\\
	 	\leq & \left(\frac{-\log(1-t)}{t}\right)\|f\|_1.
	 \end{align*}
 Again via \eqref{eq.disnorm} it follows that  $\|C_tf\|_1\leq \left(\frac{-\log(1-t)}{t}\right)\|f\|_1$, for every $f\in H^1$. Hence, $\|C_t\|_{H^1\to H^1}\leq \frac{-\log(1-t)}{t}$ which is \eqref{norn-Hp1}.

 Let $1<p<\infty$. Then, for every $f\in H^p$ and $r\in [0,1)$, by   Fubini's theorem and Lemma \ref{L.3.1} (because $s\mapsto \frac{f(sz)}{1-stz}$ is integrable over $[0,1]$ for each $z\in\D$) we obtain via \eqref{eq.disnorm} that
 \begin{align*}
 	M_p(r,C_tf)^p=&\frac{1}{2\pi}\int_0^{2\pi}\left|\int_0^1\frac{f(sre^{i\theta})}{1-srte^{i\theta}}ds\right|^pd\theta\leq \int_0^1\frac{1}{(1-ts)^p}\left(\frac{1}{2\pi}\int_0^{2\pi}|f(sre^{i\theta})|^pd\theta\right)ds\\
 	= &\int_0^1 \frac{1}{(1-ts)^p}M_p(rs,f)^pds\leq \left(\int_0^1\frac{1}{(1-ts)^p}ds\right)\|f\|_p^p\\
 	=& \frac{1}{t(p-1)}\left(\frac{1}{(1-t)^{p-1}}-1\right)\|f\|_p^p.
 \end{align*}
It follows that
\[
\|C_t\|_p\leq \left[\frac{1}{t(p-1)}\left(\frac{1}{(1-t)^{p-1}}-1\right)\right]^{1/p}\|f\|_p, \quad f\in H^p,
\]
from which \eqref{norn-Hp} follows.
\end{proof}

\begin{remark}\label{BetterE}\rm
For each $1\leq p<\infty$ and  $t\in (0,1)$ 	the  estimate of $\|C_t\|_{H^p\to H^p}$ given in \eqref{norn-Hp1} and \eqref{norn-Hp} is better than the upper estimate given in \eqref{NO}. Indeed, for $p=1$ we have $\frac{-\log(1-t)}{t}<\frac{1}{1-t}$ for every $t\in (0,1)$, as  was shown in \cite[Example 2.2]{ABR-NN}. For a fixed $1<p<\infty$,  observe that
\begin{equation}\label{eq.Stima}
\left[\frac{1}{t(p-1)}\left(\frac{1}{(1-t)^{p-1}}-1\right)\right]^{1/p}<\frac{1}{1-t},\quad t\in (0,1),
\end{equation}
if and only if
\[
\frac{1}{t(p-1)}\left(\frac{1}{(1-t)^{p-1}}-1\right)<\frac{1}{(1-t)^p},\quad t\in (0,1).
\]
Therefore, the inequality \eqref{eq.Stima} is satisfied if and only if
\[
\frac{1-(1-t)^{p-1}}{t(p-1)}\cdot\frac{1}{(1-t)^{p-1}}<\frac{1}{(1-t)^p},\quad t\in (0,1),
\]
that is, if and only if
 \begin{equation}\label{eq.Sti}
 \frac{1-(1-t)^{p-1}}{t(p-1)}<\frac{1}{1-t},\quad t\in (0,1).
 \end{equation}
To show the validity of \eqref{eq.Sti} it suffices to establish, for each $\alpha>0$,  that
\begin{equation}\label{eq.Sti2}
\frac{1-(1-t)^\alpha}{\alpha t}<\frac{1}{1-t},\quad t\in (0,1).
\end{equation}
So, fix $\alpha\geq 0$ and define $\gamma(t)=[1-(1-t)^\alpha](1-t)-\alpha t$, for $t\in [0,1]$. The function $\gamma$ is continuous in $[0,1]$ and differentiable in $(0,1)$. Furthermore, $\gamma(0)=0$ and $\gamma(1)=-\alpha<0$. On the other hand, for each $t\in (0,1)$, we have
\begin{align*}
\gamma'(t)&=\alpha (1-t)^{\alpha-1}(1-t)-[1-(1-t)^\alpha]-\alpha\\
&=\alpha(1-t)^\alpha-1+(1-t)^\alpha -\alpha\\
&=(\alpha+1)[(1-t)^\alpha-1].
\end{align*}
Since $\alpha>0$ and  $0<(1-t)^\alpha<1$ for $t\in (0,1)$, it follows that   $(1-t)^\alpha-1<0$ for $t\in (0,1)$. Accordingly,  $\gamma'(t)<0$ for every $t\in (0,1)$. This means that the function $\gamma$ is decreasing in $[0,1]$ and hence, $\gamma(t)<\gamma(0)=0$ for every $t\in (0,1)$. So, we can conclude that
\[
\gamma(t)=[1-(1-t)^\alpha](1-t)-\alpha t<0, \quad t\in (0,1),
\]
that is,
\[
[1-(1-t)^\alpha](1-t)<\alpha t, \quad t\in (0,1).
\]
The previous inequality implies that \eqref{eq.Sti2} is valid. This completes the proof that \eqref{eq.Sti} is valid. Hence, also \eqref{eq.Stima} is valid.
\end{remark}

In order to establish the compactness of the operators  $C_t\in \cL(H^p)$ we need to introduce some additional operators and  preliminaries.
	
For each $1\leq p\leq\infty$,  define
\[
H_0^p:=\{f\in H^p\, :\, f(0)=0\}.
\]
Accordingly, $H_0^p=\Ker(\delta_0)$ is a closed 1-codimensional subspace of $H^p$. Moreover,  $A_0(\D):=\{f\in A(\D):\, f(0)=0\}$ is a closed 1-codimensional subspace of $A(\D)$.

Consider now the operator  $S\colon H(\D)\to H(\D)$  defined by
\[
(Sf)(z):=zf(z),\quad f\in H(\D),\ z\in \D,
\]
which belongs to $\cL(H(\D))$ and is called
 the \textit{forward shift}. The following Lemma \ref{P_Hp} concerning the forward and backward shift operators is certainly known. Observe that the operator $S^{-1}$ in the proof of Lemma \ref{P_Hp} below is the restriction to the closed subspace $H_0^p$ of $H^p$ of the \textit{backward shift} operator given by $(Bf)(z):=(f(z)-f(0))/z$, for $ z \in \D$. This operator and its invariant subspaces on Hardy spaces have been thoroughly investigated  in \cite{CiRo}. We present a formulation  of Lemma \ref{P_Hp} which is useful for our purposes; a proof is included for the sake of completeness.

\begin{lemma}\label{P_Hp} Let  $1\leq p\leq\infty$. Then $S\in \cL(H^p)$ with $\|S\|_{H^p\to H^p}=1$ and $S(H^p)=H^p_0$. Moreover, the operator $S$ is injective with the  inverse operator $S^{-1}\colon S(H^p)\to H^p$   continuous and satisfying  $\|S^{-1}\|_{S(H^p)\to H^p}=1$.
	
	The operator $S\in \cL(A(\D))$ with operator norm $\|S\|_{A(\D)\to A(\D)}=1$ and range $S(A(\D))=A_0(\D)$. Moreover, $S$ is injective with the  inverse operator $S^{-1}\colon S(A(\D))\to A(\D)$   continuous and satisfying  $\|S^{-1}\|_{S(A(\D))\to A(\D)}=1$.
\end{lemma}

\begin{proof}
Consider first the case $1\leq p<\infty$. Given $f\in H^p$ and $r\in [0,1)$ we have
\[
M_p(r,Sf)^p=\frac{1}{2\pi}\int_0^{2\pi}|re^{i\theta} f(re^{i\theta})|^p\,d\theta\leq \frac{1}{2\pi}\int_0^{2\pi}| f(re^{i\theta})|^p\,d\theta=M_p(r,f)^p.
\]
Accordingly, $\|Sf\|_p\leq \|f\|_p$, which implies that $S\in \cL(H^p)$ and $\|S\|_{H^p\to H^p}\leq 1$. Moreover,  for every $n\in\N$, observe that
\[
\|z^n\|_p=\sup_{0\leq r<1}M_p(r,z^{n})=\sup_{0\leq r<1}\left(\frac{1}{2\pi}\int_0^{2\pi}|(re^{i\theta})^n|^p\, d\theta\right)^{1/p}=\sup_{0\leq r<1}r^n=1
\]
and that $S(z^n)=z^{n+1}$, from which we can conclude that $\|S\|_{H^p\to H^p}=1$.

Clearly, $S$ is injective and satisfies $S(H^p)\subseteq H^p_0$.

To show that $S(H^p)=H_0^p$ and that the inverse operator $S^{-1}\colon S(H^p)\to H^p$  is continuous, we proceed as follows.

Given $g\in H^p_0$, let $f(z):=\frac{g(z)}{z}=\frac{g(z)-g(0)}{z-0}$, for $z\in \D$. Clearly  $f\in H(\D)$ as $z=0$ is a removable singularity of $f$ by setting $f(0):=g'(0)$. On the other hand, for every  $r\in (0,1)$, we have
\begin{equation*}
	M_p(r,f)^p=\frac{1}{2\pi}\int_0^{2\pi}|f(re^{i\theta})|^p\,d\theta=\frac{1}{2\pi}\int_0^{2\pi}\frac{|g(re^{i\theta})|^p}{r^p}\,d\theta=\frac{1}{r^p}M_p(r,g)^p
%	\\&\leq \left(\frac{k+1}{k}\right)^p \frac{1}{2\pi}\int_0^{2\pi}|g(re^{i\theta})|^p\,d\theta=\left(\frac{k+1}{k}\right)^p M_p(r,g)^p.
\end{equation*}
It follows from \eqref{eq.disnorm}
%that $\|f\|_p\leq  \frac{k+1}{k}\|g\|_p$ for every $k\in\N$. This implies
 that  $f\in H^p$ with $g=Sf$ and $\|f\|_p\leq \|g\|_p$. Hence,  $S(H^p)=H_0^p$ (in particular $S(H^p)$ is closed in $H^p$) and  the inverse operator $S^{-1}\colon S(H^p)\to H^p$  is continuous. Moreover,  $\|S^{-1}\|_{S(H^p)\to H^p}=1$ because $\|S^{-1}g\|_p=\|f\|_p\leq \|g\|_p$ and the function $g_0(z):=z$, for $z\in\D$, belongs to $H^p_0$ and satisfies $\|g_0\|_p=1$ with $S^{-1}g_0=\mathbbm{1}$.

The case $p=\infty$ follows along the same lines.  We only observe that
\[
\|Sf\|_\infty=\sup_{z\in\D}|zf(z)|\leq \sup_{z\in\D}|f(z)|=\|f\|_\infty, \quad f\in H^\infty,
\]
and that $\|S\mathbbm{1}\|_\infty=\|g_0\|_\infty=1$. So, $S\in \cL(H^\infty)$ and $\|S\|_{H^\infty\to H^\infty}=1$.

Moreover, if $g\in H^\infty_0$, then the function $f(z):=\frac{g(z)}{z}$, for $z\in \D\setminus\{0\}$, and $f(0):=g'(0)$ belongs to $H^\infty$ (see above) and satisfies, for each $k\in\N$ and  $r\in (\frac{k}{k+1},1)$, the inequality
\[
M_\infty(r,f)=\max_{|z|=r}\left|\frac{g(z)}{z}\right|\leq \frac{k+1}{k}M_\infty(r,g).
\]
It follows that  $\|f\|_\infty\leq \frac{k+1}{k}\|g\|_\infty$ for every $k\in\N$. This implies that  $f\in H^\infty$ with $g=Sf$ and  $\|f\|_\infty\leq \|g\|_\infty$. Now proceed as for $H^p$ with $1\leq p<\infty$.

Concerning $A(\D)$, note that it is a closed, invariant subspace of $S\colon H^\infty\to H^\infty$ and so $S\in \cL(A(\D))$ with $\|S\|_{A(\D)\to A(\D)}\leq 1$. Since $\mathbbm{1}\in A(\D)$ satisfies $\|S\mathbbm{1}\|_\infty=\|g_0\|_\infty=1$, it follows that $\|S\|_{A(\D)\to A(\D)}= 1$, Clearly $S(A(\D))\subseteq A_0(\D)$. As for $H^\infty$ it can be verified that $S(A(\D))=A_0(\D)$, that $S$ is injective and that $\|S^{-1}\|_{S(A(\D))\to A(\D)}= 1$.
\end{proof}

\begin{remark}\label{R.3.6}\rm
It follows from $\|S\|_{H^p\to H^p}=1$, for each $1\leq p\leq \infty$, that  the operator $S\colon H^p\to H^p_0$ also satisfies  $\|S\|_{H^p\to H^p_0}=1$.
\end{remark}

We now investigate  a further class of operators.
For a fixed  $g\in H(\D)$, let us consider the operators  $ V_g\colon H(\D)\to H(\D)$ and $ T_g\colon H(\D)\to H(\D)$  defined by
\begin{equation}\label{eq.Volterra1}
	(V_gf)(0):=f(0),\ 	(V_gf)(z):=\frac{1}{z}\int_0^z f(\xi)g'(\xi)\,d\xi,\quad f\in H(\D),\ z\in\D\setminus\{0\},
\end{equation}
and
\begin{equation}\label{eq.Volterra2}
	(T_gf)(z):=\int_0^z f(\xi)g'(\xi)\,d\xi, \quad f\in H(\D),\ z\in\D.
\end{equation}
Note that $(T_gf)(0)=0$ for each $f\in H(\D)$. The operators $V_g$ and $T_g$   are called \textit{Volterra-type}  operators. Both operators $V_g$ and $T_g$ act continuously in $H(\D)$.
They have been investigated on different spaces of holomorphic  functions by many authors. We refer to \cite{Cont-a,Si} and the references therein.

The important connection to this paper is that the generalized Cesàro operators  $C_t\in \cL(H(\D))$, for $t\in (0,1)$, are Volterra type operators of the kind $V_g$ for suitable functions $g$. Indeed,  fix $t\in (0,1)$. Given  $f\in H(\D)$,  by \eqref{eq.formula-int} we have that $(C_tf)(0)=f(0)$ and
\begin{align*}
	(C_tf)(z) &=\frac{1}{z}\int_0^z\frac{f(\xi)}{1-t\xi}\,d\xi=\frac{1}{z}\int_0^zf(\xi)\frac{1}{1-t\xi}\,d\xi\\
	&=\frac{1}{z}\int_0^z f(\xi)\left(\frac{-\log (1-t\xi)}{t}\right)'\,d\xi, \quad z\in\D\setminus\{0\},
\end{align*}
where the function
\begin{equation}\label{delta}
g_t(z):=\frac{-\log (1-tz)}{t}=\sum_{n=0}^\infty\frac{t^n}{n+1}z^n
\end{equation}
 is holomorphic on $B(0,1/t)$ with $1/t>1$. Since $\overline{\D}\subseteq B(0,1/t)$, we can conclude that $g_t\in A(\D)\subseteq H(\D)$. Accordingly,
\begin{equation}\label{eq.UguaL}
C_tf=V_{g_t}f, \quad f\in H(\D).
\end{equation}

Topological properties such as continuity and (weak) compactness of the operators $V_g$ and $T_g$, for $g\in H(\D)$, when acting in $H^p$ are related to each other as the following result shows.

\begin{lemma}\label{L-Hp} Let $g\in H(\D)$ and $1\leq p\leq \infty$. Then $V_g\colon H^p\to H^p$ is continuous (resp. compact, resp. weakly compact) if and only if $T_g\colon H^p\to H^p$ is continuous (resp. compact, resp. weakly compact). In the case of continuity we have $\|T_g\|_{H^p\to H^p}=\|V_g\|_{H^p\to H^p}$.
	\end{lemma}

	\begin{proof} Fix $1\leq p\leq \infty$. Assume first that  $V_g\in \cL(H^p)$. By Lemma \ref{P_Hp} the forward shift operator  $S\in \cL(H^p)$ and so also $T_g=S\circ V_g$  belongs  to $\cL(H^p)$. Conversely, suppose that $T_g\colon H^p\to H^p$ is continuous. Clearly, $T_g(H^p)\subseteq H^p_0$ and $T_g\colon H^p\to H^p_0$ is also continuous. Since  the inverse operator $S^{-1}\colon H^p_0\to H^p$  exists and is continuous, the operator $V_g=S^{-1}\circ T_g$ belongs  to $\cL(H^p)$. The proof for the compactness (resp.  weak compactness) follows along the same lines.

From the fact that  $\|S\|_{H^p\to H^p_0}=\|S^{-1}\|_{H^p_0\to H^p}=1$ (cf.  Lemma \ref{P_Hp} and Remark \ref{R.3.6}), together with  $T_g=S\circ V_g$ and  $V_g=S^{-1}\circ T_g$, it follows that $\|T_g\|_{H^p\to H^p}=\|V_g\|_{H^p\to H^p}$.
	\end{proof}

 The following definitions play an important role; see \cite{Z} for more details.
The space $BMOA$ consists of all functions $f\in  H^2$ such that
  	\[
  |f(0)|+\sup_{a\in\D}\|f\circ \phi_a-f(a)\|_2<\infty,
  	\]
  	where $\phi_a$, for $a\in\D$, is the family of M\"obius automorphisms of $\D$ given by $\phi_a(z):=\frac{z-a}{1-\overline{a}z}$,  for $z\in\D$. The space $VMOA $ consists of  all functions  $f\in  BMOA$ satisfying
  \[
  \lim_{|a|\to 1}\|f\circ \phi_a-f(a)\|_2=0.
  \]
  The space $VMOA$ is the closure of the polynomials in $BMOA$, \cite[Theorem 5.5]{Gir}. In particular, $H^\infty\subseteq BMOA$ and $A(\D)\subseteq VMOA$, \cite[Theorem 5.5 and Remark 5.2]{Gir}.

The following result collects together various facts concerning the operators $T_g$ when they act in the $H^p$-spaces.

\begin{theorem}\label{T.38} Let $1\leq p<\infty$.
	\begin{itemize}
		\item[\rm (i)] The operator $T_g\colon H^p\to H^p$ is compact if and only if $g\in VMOA$. In particular, if $g\in A(\D)$, then both $T_g\colon H^p\to H^p$ and $V_g=S^{-1}\circ T_g\colon H^p\to H^p$ are compact.
		\item[\rm (ii)] Let $t\in (0,1)$. The generalized Cesàro operator $C_t\colon H^p\to H^p$ is compact.
		Hence, also the operator $S_t:=T_{g_t}=S\circ C_t\colon H^p\to H^p$ is  compact, where $g_t$ is given by \eqref{delta}.
	\end{itemize}
	\end{theorem}

For part (i) of Theorem \ref{T.38} we refer to  \cite{AleSi1} and, for all $p>0$, to \cite{AleCi};  see also \cite[Corollary 4.2]{Si}. Part (ii) of Theorem \ref{T.38} follows from part (i) after recalling (see \eqref{eq.UguaL}), for each $t\in [0,1)$, that $C_t=V_{g_t}$ with $g_t\in A(\D)$.

Compactness criteria for the operators $T_{g}$  on both $H^\infty$ and on $A(\D)$ are formulated in the following result.

\begin{theorem}\label{T-w} {\rm (i)} Let $g\in H(\D)$ and $T_g\colon H^\infty\to H^\infty $ be weakly compact. Then $g\in A(\D)$.
	
{\rm (ii)}	 For $g\in H(\D)$ the following statements are equivalent.
	\begin{itemize}
		\item[\rm (a)] $T_g\colon H^\infty\to H^\infty$ is compact.
		\item[\rm (b)] $T_g\colon A(\D)\to A(\D)$ is compact.
		\item[\rm (c)] $T_g\colon H^\infty\to A(\D)$ is compact.
		\item[\rm (d)] $T_g\colon A(\D)\to H^\infty$ is compact.
	\end{itemize}
	If either one of {\rm (a)--(d)} holds, then necessarily  $g\in A(\D)$.
\end{theorem}

We point out that parts (i) and (ii) of Theorem \ref{T-w} are, respectively, Theorem 1.4 and Theorem 1.7 in  \cite{Cont-a}. The statement in part (ii) that $g\in A(\D)$ follows from part (i).

%Recall from \eqref{eq.UguaL} that, for each  $t\in (0,1)$, the generalized Cesàro operator satisfies  $C_t=V_{g_t}$ with $g_t\in A(\D)$. This has the following consequence.

\begin{prop} Let $t\in [0,1)$. Then $C_t(H^\infty)\subseteq A(\D)$.
\end{prop}

\begin{proof}
	Fix $t\in [0,1)$. By  Proposition \ref{Compact} the operator $C_t=V_{g_t}\colon H^\infty\to H^\infty$ is compact and hence, by Lemma \ref{L-Hp} the operator $T_{g_t}\colon H^\infty\to H^\infty$ is compact, where $g_t$ is given in \eqref{delta}. So, Theorem \ref{T-w}(ii) ensures that the operator $T_{g_t}\colon A(\D)\to A(\D)$  is also compact and satisfies $T_{g_t}(H^\infty)\subseteq A(\D)$. Since $(T_{g_t}f)(0)=0$, for every $f\in H(\D)$, actually $T_{g_t}(A(\D))\subseteq A_0(\D)$. Moreover, Lemma \ref{P_Hp} shows that $S^{-1}\colon A_0(\D)\to A(\D)$ continuously. Hence,   $C_t=V_{g_t}=S^{-1}\circ T_{g_t}\colon A(\D)\to A(\D)$ is compact and $C_t(H^\infty)=V_{g_t}(H^\infty)=S^{-1}(T_{g_t}(H^\infty))\subseteq A(\D)$.
\end{proof}

We  are now able to calculate the spectrum of the generalized Cesàro operator $C_t$ and of the operator $S_t=S\circ C_t$, for $t\in [0,1)$, when they  act in $H^p$.

\begin{prop}\label{Spectrum-Hp} Let $1\leq p<\infty$. For each $t\in [0,1)$ the  spectra of $C_t\in \cL(H^p)$  are given by
	\begin{equation}\label{Sp-pt-Hp}
		\sigma_{pt}(C_t;H^p)=\left\{\frac{1}{m+1}\,: m\in\N_0\right\},
	\end{equation}
	and
	\begin{equation}\label{Sp-Hp}
		\sigma(C_t;H^p)=\left\{\frac{1}{m+1}\,: m\in\N_0\right\}\cup\{0\}.
	\end{equation}
\end{prop}

\begin{proof} Let $1\leq p<\infty$ and $t\in [0,1)$  be fixed.
Recall,  by \cite[Proposition 3.7]{ABR-NN}, that  the point spectrum of  $C_t\in \cL(H(\D))$ is given by $\sigma_{pt}(C_t;H(\D))=\{\frac{1}{m+1}\,:\, m\in\N_0\}$. Since $A(\D)\subseteq H^p\subseteq H(\D)$ with continuous inclusions, we obtain that   $\sigma_{pt}(C_t;A(\D))\subseteq \sigma_{pt}(C_t;H^p)\subseteq \sigma_{pt}(C_t;H(\D))$. Then Proposition \ref{Spectrum} implies the validity of \eqref{Sp-pt-Hp}.
	
Finally, \eqref{Sp-Hp} follows from the fact that $C_t$ is a  compact operator  on  $H^p$; see Theorem   \ref{T.38}(ii).
\end{proof}

\begin{prop}\label{P_spectrumS}Let $1\leq p<\infty$.
	For each $t\in [0,1)$ the  spectra of $S_t\in \cL(H^p)$  are given by
	\begin{equation}\label{Sp-pt-Hp-s}
		\sigma_{pt}(S_t;H^p)=\emptyset; \
		\sigma_r(S_t;H^p)=\{0\};\ \sigma_c(S_t;H^p)=\emptyset.
	\end{equation}
In particular, $\sigma(S_t;H^p)=\{0\}$.
\end{prop}

\begin{proof} Let $1\leq p<\infty$ and $t\in (0,1)$ be given. For a fixed $\lambda\in\C\setminus\{0\}$ the equation
	$\lambda f-S_tf=0$ yields, after differentiation, that $\lambda f'(z)=\frac{f(z)}{1-tz}$ for $z\in\D$, which has the solutions
	 $f(z)=A(1-tz)^{-1/(t\lambda)}$ for constants $A\in\C$. Moreover, $S_tf=\lambda f$ together with $S_t(H^p)\subseteq H^p_0$ implies that $f(0)=0$, that is, $A=0$ and hence, $f\equiv 0$.
	
	For $t=0$ and  $\lambda\in\C\setminus\{0\}$ we consider the solutions  of $\lambda f-S_0f=0$, for $f\in H^p$. From the definition of $S_0$ this equation reduces  to $\int_0^z f(\xi)d\xi=\lambda f(z)$. Differentiating yields $f'(z)=\frac{1}{\lambda}f(z)$ and so $f(z)=Be^{z/\lambda}$ for some $B\in\C$. Since $(S_0h)(0)=0$ for  every $h\in H(\D)$, it follows that $B=0$. Hence,  $f\equiv 0$ is the only solution.
	
	Finally, for $\lambda=0$ the equation $S_tf=0$ implies $f\equiv 0$ as both $S\in \cL(H^p)$ and $C_t\in\cL(H^p)$ are injective.

	So, we have established that  $\sigma_{pt}(S_t;H^p)=\emptyset$.
	
	Since the operator $S_t\in \cL(H^p)$ is compact (cf. Theorem \ref{T.38}(ii)), it follows that $\sigma(S_t;H^p)=\sigma_{pt}(S_t;H^p)\cup \{0\}=\{0\}$. Moreover, $S_t$ is injective and $S_t(H^p)=(S\circ C_t)(H^p)\subseteq S(H^p)\subseteq H_0^p$ show that $S_t(H^p)$ is not dense in $H^p$. Hence, $0\in \sigma_r(S_t;H^p)$. The proof of \eqref{Sp-pt-Hp-s} is thereby complete.
	\end{proof}

Our final result establishes certain linear dynamic features of   $C_t$ and $S_t$.

\begin{prop}\label{Dyn-Hv-Hp} Let $1\leq p<\infty$. For each $t\in [0,1)$  both of the operators $C_t\in \cL(H^p)$ and $S_t\in \cL(H^p)$ are power bounded, uniformly mean ergodic and fail to be supercyclic.
\end{prop}

\begin{proof}
	Fix $t\in [0,1)$.  The operator $C_t$ is   compact in    $H^p$  by Theorem \ref{T.38}(ii). Therefore, the transpose operator $C_t'\in \cL((H^p)')$, which is also compact,  has the same non-zero eigenvalues as $C_t$; see \cite[Theorem 9.10-2(2)]{Ed}. In view of Proposition \ref{Spectrum-Hp} it follows that $\sigma_{pt}(C_t';(H^p)')=\{\frac{1}{m+1}\,:\, m\in\N_0\}$. So, we can apply \cite[Proposition 1.26]{B-M} to conclude that $C_t$ is not supercyclic on  $H^p$.
	
	Arguing as in the proof of Proposition \ref{Dyn-Hv}
	one shows that   ${\rm Im}(I-C_t)\cap\Ker (I-C_t)=\{0\}$. On the other hand,
	Proposition \ref{Spectrum-Hp} implies that $1\in \sigma(C_t; H^p)=\{\frac{1}{m+1}\,;\, m\in\N_0\}\cup\{0\}$. Consequently, for $\delta=\frac{1}{2}$, all the assumptions of Theorem \ref{Th-ABR} are satisfied; see also Remark \ref{R.12}. So, we can conclude that $C_t$ is power bounded and uniformly mean ergodic on   $H^p$.

The operator $S_t\in \cL(H^p)$ is also power bounded and uniformly mean ergodic. Indeed, Proposition \ref{P_spectrumS} implies that its spectral radius $r(S_t;H^p)=0$.
 On the other hand, it is  known that  also $r(S_t;H^p)=\lim_{n\to \infty}(\|S_t^n\|_{H^p\to H^p})^{1/n}$; see pp.234-235 of \cite{Ru} for the unital Banach algebra $\mathcal{A}:=\cL(H^p)$. Accordingly, $\|S_t^n\|_{H^p\to H^p}\to 0$ as $n\to\infty$ and so $S_t$ is surely power bounded.

 Given $f\in H^p$, its projective orbit satisfies
 \[
 \{\lambda S_t^nf:\ \lambda\in\C,\ n\in\N_0\}\subseteq {\rm span}(\{f\})\cup  \{\lambda S_t^nf:\ \lambda\in\C,\ n\in\N\}.
 \]
 But, $S_t^n(H^p)\subseteq H^p_0$ for all $n\in\N$ and so
 \begin{equation}\label{A}
 \{\lambda S_t^nf:\ \lambda\in\C,\ n\in\N_0\}\subseteq {\rm span}(\{f\})\cup  H^p_0.
 \end{equation}
 Since both   ${\rm span}(\{f\})$ and $H^p_0$ are proper closed subsets of $H^p$, also their union is a proper closed subset of $H^p$. It follows from \eqref{A} that the projective orbit of $f$ cannot be dense in $H^p$. Since $f\in H^p$ is arbitrary, we can conclude that $S_t$ is not  supercyclic.
\end{proof}

\vspace{.1cm}

\textbf{Acknowledgements.} The research of A.A. Albanese was partially supported by GNAMPA of INDAM.
The research of J. Bonet was partially supported by the project
PID2020-119457GB-100 funded by MCIN/AEI/10.13039/501100011033 and by
``ERFD A way of making Europe''.

 \vspace{.5cm}

\textbf{Data availability.} Not applicable.

%\bigskip
\bibliographystyle{plain}

\end{document}